\documentclass{amsart}

\usepackage{amsmath}
\usepackage{tikz}
\usepackage{booktabs}

\usetikzlibrary{shapes.geometric}

\usepackage{amssymb}
\usepackage{amscd}

\title[(except possibly for finitely many points)]{Everything is
  illuminated }

\author{Samuel Leli\`evre}
\address{Laboratoire de math\'ematique d'Orsay, Umr 8628 Cnrs/Universit\'e 
Paris-Sud, B\^at.~425, campus Orsay, 91405 Orsay cedex, France,
{\tt samuel.lelievre@math.u-psud.fr}}

\author{Thierry Monteil}
\address{
LIRMM - CNRS (UMR 5506),  
161 rue Ada 
34392 Montpellier cedex 5
France {\tt thierry.monteil@lirmm.fr} 
}

\author{Barak Weiss}

\address{School of Math. Sciences, Tel Aviv University, Israel,
{\tt barakw@post.tau.ac.il}}

\newif\ifdraft\drafttrue

\draftfalse

\font\sn = cmssi8 scaled \magstep0

\newcommand\name[1]{\label{#1}{\ifdraft{\sn [#1]}\else\ignorespaces\fi}}

\newcommand\eq[2]{{\ifdraft{\ \tt [#1]}\else\ignorespaces\fi}\begin{equation}\label{eq:
#1}{#2}\end{equation}}

\newcommand {\equ}[1]     {\eqref{eq: #1}}

\renewcommand{\phi}{\varphi}
\newcommand{\bc}{\operatorname{bc}}

\newcommand{\Q}{{\mathbb {Q}}}
\newcommand{\BCP}{{\mathrm {BCP}}}

\newcommand{\BB}{\operatorname{BB}}

\newcommand{\R}{{\mathbb{R}}}

\newcommand{\Z}{{\mathbb{Z}}}

\newcommand{\C}{{\mathbb{C}}}

\newcommand{\HH}{{\mathcal{H}}}

\newcommand{\N}{{\mathbb{N}}}

\newcommand{\SL}{\operatorname{SL}}

\newcommand{\cL}{{\mathcal L}}

\newcommand{\df}{{\, \stackrel{\mathrm{def}}{=}\, }}

\newcommand{\FF}{{\mathcal{F}}}

\newcommand{\x}{{\bf x}}

\newcommand{\til}{\widetilde}

\newcommand{\supp}{{\rm supp}}

\newcommand{\sm}{\smallsetminus}

\newcommand\Aff{\mathrm{Aff}}

\newcommand{\cF}{{\mathcal F}}
\newcommand{\cH}{{\mathcal H}}

\newcommand{\REL}{\operatorname{REL}}

\newcommand{\GREL}{G\oplus\REL}

\newcommand {\ignore}[1]  {}

\newtheorem{thm}{Theorem}

\newtheorem{lem}[thm]{Lemma}

\newtheorem{prop}[thm]{Proposition}

\newtheorem{cor}[thm]{Corollary}

\newtheorem{remark}[thm]{Remark}

\begin{document}

\begin{abstract}
We study geometrical properties of translation surfaces: the
finite blocking property, bounded blocking property, and illumination
properties. These are elementary
properties which can be fruitfully 
studied using the dynamical behavior of the $\SL(2,\R)$-action on the
moduli space of translation surfaces. We characterize surfaces with
the finite blocking property and bounded blocking property, completing
work of the second-named author \cite{ThierryBP}. Concerning the
illumination problem, we also
extend results of Hubert-Schmoll-Troubetzkoy
\cite{HST}, removing the hypothesis that the surface in question is a
lattice surface, thus settling a conjecture of \cite{HST}. Our results crucially rely on 
the recent breakthrough results of Eskin-Mirzakhani \cite{EsMi} and 
Eskin-Mirzakhani-Mohammadi \cite{EsMiMo}, and on related results of Wright \cite{Wright}.
\end{abstract}

\maketitle
\section{Introduction}
A {\em translation surface} is a finite union of polygons, glued along
parallel edges by translations, up to a cut and paste
equivalence. These structures arise in the study of billiards,
interval exchange transformations, and various problems in group
theory and geometry. See \cite{MT, Zorich} for comprehensive
introductions and detailed definitions. 
The purpose of this paper is to apply recent breakthrough results of
Eskin-Mirzakhani \cite{EsMi} and 
Eskin-Mirzakhani-Mohammadi \cite{EsMiMo}, on the dynamics of a group
action on the moduli space of translation surfaces, to some 
elementary geometrical questions concerning translation surfaces. We begin with
some definitions. 

A pair of points $(x,y) \in M \times M$ is {\em finitely blocked} if
there exists a finite set $B \subset M$ which does not contain $x$ or
$y$ and intersects every straight-line trajectory connecting $x$ and
$y$. A set $B$ with this property is called a {\em blocking set} for
$(x, y)$, and the minimal
cardinality of a blocking set is called the {\em blocking
  cardinality of $(x,y)$} and is denoted by $\bc(x,y)$. 
A translation surface $M$ has the \emph{blocking property} if any
pair $(x,y) \in M \times M$ is finitely blocked, and the \emph{bounded
  blocking property} if 
there is a number $n$ such that any pair $(x,y) \in M \times M$ is
finitely blocked with blocking cardinality at most $n$. 
If $x$ and $y$ are finitely blocked with blocking cardinality zero,
that is, if there is no straightline path on $M$ from $x$ to $y$, then
we say that $x$ and $y$ {\em do not illuminate each other}. 
A translation surface $M$ is a \emph{torus cover} if there is a
surjective translation map from $M$ to a torus (the singularities of
$M$ may project to one or several points on the torus). 
Equivalently
(see e.g.\ \cite{ThierryBP}), the 
subgroup of $\R^2$ generated by holonomies of absolute periods on $M$
is discrete.  

Our first result settles a question of the second-named author, see
  \cite{ThierryBP, ThierryPP}. 

\begin{thm}\name{thm: BP}
For a translation surface $M$, the following are equivalent:
\begin{enumerate}
\item \name{enum:torus-cover}
$M$ is a torus cover.
\item \name{enum:blocking}
$M$ has the blocking property.
\item \name{enum:open-blocking}
There is an open set $U \subset M \times M$ such that any pair of
points in $U$ is finitely blocked.
\item \name{enum:bounded-blocking}
$M$ has the bounded blocking property.
\end{enumerate}
\end{thm}

Hubert, Schmoll and Troubetzkoy \cite{HST} have constructed an example of a
translation surface $M$ which is not a torus cover, and in which
there are infinitely many pairs of points which do not illuminate each
other. In fact, there is an involution $\tau: M \to M$ such that for
any $x \in M$, there is no straight line between $x$ and
$\tau(x)$. See \S \ref{related example2} for
similar examples. This
shows that in (\ref{enum:open-blocking}), it is not 
enough to suppose that $U$ is infinite. 

Our second result concerns questions of illumination. The classical illumination
problem was first posed in the 1950's, when it was asked whether there
exists a polygonal room with a pair of points which do not
illuminate each other. First examples were found by Tokarsky
\cite{Tokarsky} and Boshernitzan (unpublished), and this raised the
question of classification and possible cardinality of pairs of points
which do not illuminate one another on translation surfaces. We refer to \cite{HST} or the
Wikipedia page {\tt
  http://en.wikipedia.org/wiki/Illumination\_problem} for a brief
history. We show:
\begin{thm}\name{thm: illumination}
For any translation surface $M$, and any point $x \in M$, the set of
points $y$ which are not illuminated by $x$ is finite.

Moreover, the set
$$
\{(x,y): x \text{ and } y \text{ do not illuminate each other} \}
$$
is the union of a finite set, and of finitely many translations
surfaces $M'$ embedded in $M \times M$, such that the projections
$p_i|M': M' \to M$ are both finite-degree covers.  
\end{thm}
Here $p_i: M \times M
\to M, \ i=1,2$ are the natural projections onto 
the first and second factors respectively. 

Theorem \ref{thm: illumination} strengthens 
results of \cite{HST}, which deal with surfaces which have a
large group of translation automorphisms. Namely, Theorem \ref{thm:
  illumination} was proved in
\cite{HST} under the additional hypothesis that $M$ is a lattice
surface, and when $M$ is a pre-lattice surface, the first assertion of
the theorem was shown, with `countable' in place of `finite' (for the
definitions see \S \ref{subsec: Veech groups}).  
The first assertion of Theorem \ref{thm: illumination} settles
\cite[Conjecture 1]{HST}. In \S \ref{section: illumination} we deduce
Theorem \ref{thm: illumination} from the more general Theorem
\ref{thm: more general}. In \S \ref{section: examples} we give
examples which elaborate on related examples given in \cite{HST}. 

A standard `unfolding' technique (see \cite{MT, Zorich}) leads to the
following result, which justifies the title of this paper. 
\begin{cor}\name{cor: billiard illumination}
Let $P$ be a rational polygon. Then for any $x \in P$ there are at
most finitely many points $y$ for which there is no geodesic
trajectory between $x$ and $y$. 
\end{cor}

There is a moduli space $\cH$ parameterizing all translation surfaces
sharing some topological data, and this space is equipped with an
action of the group $G \df \SL(2,\R)$. The 
breakthrough work \cite{EsMi, EsMiMo} has made it
possible to analyze the dynamics of this action in great detail. Our
analysis depends crucially on this work, as well as on additional work
of Wright \cite{Wright}. 

We note that the crucial feature which make our analysis possible is
that the geometric properties we consider give rise to subsets of
$\cH$ which are closed and $G$-invariant. It has long been known that
a detailed understanding of the $G$-action would shed light on the
illumination problem, as well as on many
similar `elementary' problems. For more papers applying the
dynamics of the $G$-action to the analysis of closed and $G$-invariant
geometrical properties of translation surfaces, see 
\cite{Veech_hayashibara, Vorobets, ThierryBP, ThierryPP, HST, nst, toronto, noconvex}.

\subsection{Acknowledgements} 
This works relies on deep results of Alex Eskin, Maryam Mirzakhani and
Amir Mohammadi, and is influenced by Alex Eskin's vision that the geometrical
problems considered here can be solved via ergodic theory. We are also
grateful to Erwan Lanneau, Duc-Manh Nguyen, John Smillie and Alex
Wright for 
helpful discussions. 
This research was supported by the ANR projet blanc GEODYM and
European Research Council grant DLGAPS 279893.

\ignore{

 Say that it
enables the classification of translation surfaces with fixed
geometric properties, as long as they correspond to {\em invariant
  closed sets} in strata of translation surfaces. 
Mention $G$-invariant properties that are known to us, and their
variants which give closed sets:
\begin{itemize}
\item
Admitting a branched cover of a certain topological type. 
\item
(Bounded) blocking property. 
\item
No (strictly) convex presentations. 
\item 
Fewer than normal disjoint cylinders. 
\item
Cylinders do not generate homology. 
\item
No small $n$-gons or virtual $n$-gons or convex $n$-gons. Similarly
with virtual $n$-gons in which the saddle connections join some
designated subset of the singularities.  
\item Bounded complete periodicity, bounded in the sense of Vorobets. 
\item No face to face presentations (a presentation of a translation
  surface is a {\em face to face polygon} if it is a polygon
  representing the surface such that
  each side of the polygon contains a point which sees its image on
  the other side). Similarly replacing `one point in each edge sees its image' with
  `one point in each edge sees another point in the same edge', or `the whole edge sees the whole
  edge'. 
\end{itemize}
}
\section{Preliminaries}\name{sec: preliminaries}
We begin by briefly recalling the definitions of translation surfaces
and strata, and refer to \cite{MT, Zorich} for more details. 
Fix a topological orientable surface $S$ of genus $g$, a finite subset $\Sigma =
\{x_1, \ldots, x_k\}$ of $S$, and non-negative integers $\alpha_1,
\ldots, \alpha_k$ so that $\sum_i \alpha_i = 2g-2.$ We allow some of the
$\alpha_i$ to be zero and require $k \neq 0$. A translation surface
$M$ of type $\vec{\alpha} = (\alpha_1, \ldots, \alpha_k)$ is a surface
$M$ homeomorphic to $S$, with $k$ labelled 
singular points $\{\xi_1, \ldots, \xi_k\}$, equipped with an
equivalence class of atlases of {\em planar charts}, i.e.\ maps 
from open subsets of $M \sm \{\xi_1, \ldots, \xi_k\}$ to $\C$, such that:
\begin{itemize}
\item
Transition maps for the charts are translations.
\item
At each $\xi_i$ the charts give rise to a cone type singularity of
angle $2\pi(\alpha_i+1)$. 
\end{itemize}
As usual two atlases are considered equivalent if their union is also an
atlas of the same type, and two translation surfaces are considered
equivalent if there is a homeomorphism from one to the others, which is a
translation in charts, and maps
the distinguished finite set $\{\xi_i\}$  of one translation surface
bijectively to the other in a way which 
respects the numbering. Note that an atlas of planar charts on $M \sm \Sigma$
naturally induces a translation structure on $(M \sm \Sigma) \times (M
\sm \Sigma)$, with charts
taking values in $\C^2$ and for which transition maps are
translations. We will call this the {\em Cartesian product translation
  structure on $M^2$}. 

The points $\xi_i$ are called {\em singularities}. Note that we have
allowed singularities with cone angle $2\pi$ (as happens when
$\alpha_i=0$). Such singularities are sometimes referred to as {\em
  marked points}. Note also that in contrast to the convention used by
some authors, our convention is that 
singularities are labeled. 

A homeomorphism $S \to M$ which
maps each $x_i$ to $\xi_i$ is called a {\em marking}. We can use a
marking and the planar charts of $M$ to evaluate the integrals
of directed paths on $S$ beginning and ending in $\Sigma$. Such an integral is
a complex number whose real and imaginary components measure respectively the total
horizontal and vertical distance travelled when moving in $M$ along
the image of the path. Denote by $\cH(\vec{\alpha})$ the set of translation surfaces of type
$\vec{\alpha}$. It is called a {\em stratum} and is equipped with a 
natural topology defined as follows.  The discussion above shows that the marking gives rise to a map 
$$
\cH(\vec{\alpha}) \to
H^1(S, \Sigma ; \C).
$$
It is known that the maps above constitute an atlas of charts which
endow $\cH(\vec{\alpha})$ with the structure of a linear orbifold. We
will call these coordinates {\em period coordinates}. With respect
to period coordinates, the change of a marking constitutes a change
of coordinates via a unimodular integral matrix, so $\cH(\vec{\alpha})$ is
naturally endowed with a Lebesgue measure and a $\Q$-structure. It is
known that each stratum has 
finitely many connected components. 

The group $G$ acts on each stratum component $\cH$ by postcomposition
of planar charts. That is, identifying the field of complex numbers
with the plane $\R^2$ in the usual way, each $g \in G$ is a linear map of $\R^2$
and we use it to replace each chart $M \supset
U \stackrel{\varphi}{ \to} \C \cong \R^2$ with the chart $g \circ \varphi: U
\to \R^2$. For each stratum component $\cH$, the subset $\cH^{(1)}$
consisting of area one surfaces is a 
sub-orbifold which in period coordinates is cut out by a quadratic
condition. It is preserved by the $G$-action, and $G$ acts ergodically
preserving a natural smooth finite measure obtained from the Lebesgue
measure by a cone construction. Given a translation surface $M$ and a
positive real number $t$, we denote by $tM$ the translation surface
obtained by multiplying all planar charts of $M$ by the scalar $t$. 

\subsection{Adding marked points}\name{subsec: adding points}
We will need some notation for the operation of covering a stratum by
a corresponding stratum with one or two additional marked points. 

Given a stratum component $\cH$, we denote by $\cH'$ the corresponding
stratum component of
surfaces with one additional marked point, and by  
$\cH''$ the corresponding stratum component of surfaces with two additional marked
points. More formally this is defined as follows. Suppose $\cH$ is a
component of $\cH(\vec{\alpha})$ where $\vec{\alpha} \df  (\alpha_1,
\ldots, \alpha_k)$ and $\Sigma \df \{x_1, \ldots, x_k\}$ is a finite
subset of cardinality $k$ in the topological surface $S$. Let
$x_{k+1}, x_{k+2}$ denote two distinct points on $S \sm \Sigma$, set
$\alpha_{k+1} = \alpha_{k+2}=0$, set
$$
\Sigma' \df \Sigma \cup 
\{x_{k+1}\}, \ \ \Sigma'' \df \Sigma' \cup \{x_{k+2}\}, \ \ \vec{\alpha}'
\df (\alpha_1, \ldots, \alpha_{k+1}), \ \ \vec{\alpha}''
\df (\alpha_1, \ldots, \alpha_{k+2}),
$$
and let $\varphi': \cH(\vec{\alpha}') \to \cH(\vec{\alpha}), \ 
\varphi'': \cH(\vec{\alpha}'') \to \cH(\vec{\alpha}')$ be the {\em
  forgetful maps} obtained by deleting the points corresponding to
$x_{k+1}, x_{k+2}$ from the domain of any planar chart. Let $\varphi
\df \varphi' \circ \varphi''$. The three maps $\varphi', \varphi'',
\varphi$ are
bundle maps for the respective bundles $\cH(\vec{\alpha}'),
\cH(\vec{\alpha}''), \cH(\vec{\alpha}'') $ with bases $\cH(\vec{\alpha}), \cH(\vec{\alpha}'), \cH(\vec{\alpha})$ and fibers $S \sm \Sigma, S \sm
\Sigma', (S \sm \Sigma)^2 \sm \Delta$ respectively ($\Delta$ is the diagonal). Finally we let $\cH', 
\cH''$ be the connected components of $\cH(\vec{\alpha}')$ and
$\cH(\vec{\alpha}'')$ covering the component $\cH$. 

One 
easily checks from the definitions that the maps $\varphi, \varphi',
\varphi''$ are $G$-equivariant, and that the fibers are linear
manifolds in period coordinates. Moreover note that the linear
structure on a fiber $\varphi'^{-1}(M) \cong S \sm \Sigma$ coincides
with the translation structure afforded by the translation charts on
$M$, and similarly, the linear structure on a fiber $\varphi^{-1}(M)
\cong (S \sm \Sigma)^2 \sm \Delta$ coincides with the Cartesian product
translation structure on $M^2$. In the sequel we will refer to
$x_{k+1}, x_{k+2}$ as the first and second marked points for the
covers $\cH'' \to \cH' \to \cH$. Note that we allow $\cH$ to contain
additional marked points.

\subsection{Recent dynamical breakthroughs}
We now state the results of \cite{EsMi, EsMiMo, Wright} mentioned 
in the introduction. This requires some terminology. We say that a subset $\cL_0 \subset \cH$ is a
{\em complex linear manifold defined over $\R$} if for each of the
charts $\cH \to H^1(S, \Sigma; \C) \cong \C^{}$ obtained by fixing a
marking, the image of $\cL_0$ is the intersection of an open set with
an affine subspace whose linear part is a $\C$-linear vector space
defined over $\R$. Note that the real dimension of a complex linear 
manifold is even. Given $\cL \subset \cH^{(1)}$, we denote 
$$
\widehat{\cL} \df \{tM': t>0, M' \in \cL\}. 
$$
If $\nu$ is a measure on $\cH$ then $\mu(A) = \nu(\{tx: x \in A, t \in
(0,1]\}$ is a measure on $\cH^{(1)}$ and we say that {\em $\mu$ is
  obtained by coning off $\nu$.} 
 We say that $\cL \subset \cH^{(1)}$ is an {\em affine invariant
 manifold} if it is $G$-invariant, is the support of an ergodic
$G$-invariant measures $\mu$, $\widehat{\cL}$ is
a complex linear manifold defined over $\R$, and $\mu$ is 
obtained by coning off Lebesgue measure on $\widehat{\cL}$. 

\begin{thm}[Eskin-Mirzakhani-Mohammadi]\name{EsMiMo} 
For each stratum component $\cH$ and each $M \in \cH^{(1)}$, the orbit
closure $\cL \df \overline{GM}$ is an affine invariant manifold. 
The collection of  affine invariant manifolds of $\cH$ obtained as
orbit-closures for the $G$-action is countable. If $\cL_n, \, n\geq 1$
is a sequence of distinct affine invariant manifolds of some dimension 
$k$ contained in $\HH$, then after passing to a subsequence, the set of accumulation
points 
$$\{M \in \HH: \exists M_n \in \cL_n \text{
  such that } M_n \to M\}$$ 
is an affine invariant manifold $\cL_{\infty}$ with $\dim \cL_{\infty} >k$ and $\{M_n\}
\subset \cL_{\infty}$. 
\end{thm}
Note that the results of \cite{EsMiMo} work for strata with marked
points, i.e.\ they allow $\alpha_i=0$.

Suppose that the number of singularities $k$ is at least  $2$. Let
$H_1(S)$ and $H_1(S, \Sigma)$ denote respectively the absolute and
relative homology groups. Then we have $H_1(S) \subset H_1(S, \Sigma)$
and we can restrict each 1-cocycle in $H^1(S, \Sigma; \C)$ to the
subspace $H_1(S)$; that is we get a natural restriction map $H^1(S,
\Sigma; \C) \to H^1(S; \C)$. The kernel REL of this map is a subspace
of $H^1(S, \Sigma; \C)$ of real dimension $2(k-1)$,  and we have a foliation
of $H^1(S, \Sigma; \C)$ by cosets of REL. Since the restriction map 
$H^1(S,
\Sigma; \C) \to H^1(S; \C)$ is topological, the space REL is
independent of a marking, that is can be used to unequivocally define
a linear foliation of $\cH(\vec{\alpha})$ using period coordinates. This
foliation of $\cH(\vec{\alpha})$ is called the {\em REL
  foliation}. The $G$-action respects the REL foliation and hence we
have a linear foliation of $\cH$ by leaves tangent to
$\mathfrak{g} \oplus \REL$, where we use $\mathfrak{g}$ to denote the tangent to the
foliation by $G$-orbits. We denote this foliation by $\GREL$. Following
\cite{Wright}, if a closed
$G$-invariant and $G$-ergodic linear manifold $\cL$ is contained in a single leaf
of the foliation $\GREL$, we say that it is {\em of cylinder rank
  one}. A translation surface $M$ is \emph{completely periodic} if in any
cylinder direction on $M$ there is a complete cylinder
decomposition.

\begin{thm}[Wright \cite{Wright}, Theorems 1.5 and 1.6]
\name{prop: Wright equivalences}
A linear manifold $\cL$ as above is of cylinder rank one if and only if any surface in  $
\cL$ is completely periodic. 
\end{thm}

We will need the following Lemma. Note that its assertion would be trivial if the fiber
of $\varphi$ were compact. 

\begin{lem}\name{lem: projection closed}
Let $M \in \cH$ and $M'' \in \varphi^{-1}(M) \subset \cH''$. Let $\cL
\df \overline{GM}$ and $\cL'' \df \overline{GM''}$. Then
$\varphi|_{\cL''}$ is an open mapping and hence $\dim \varphi(\cL'') =
\dim \cL.$ 
\end{lem}
\begin{proof}
According to \cite{EsMiMo}, there
are Borel probability measures $\mu, \, \mu''$ on $\cH, \, \cH''$
respectively such that $\cL = \supp \mu, \, \cL''=\supp \mu''$. We
first claim that 
$\mu =
\varphi_*\mu''$. To this end note that Theorems 2.6 and 
2.10 in \cite{EsMiMo} provide an averaging method converging to $\mu, \mu''$;
that is, in both of these theorems, one finds probability measures
$\nu_T$ on $G$, such that for any continuous compactly supported
functions $f, f''$ on $\cH$ and $\cH''$ respectively, 
$$
\int_{G} f(gM) d\nu_T(g) \to_{T \to \infty} \int_{\cH} f d\mu
\text{ (resp., }\int_{G} f(gM'') d\nu_T(g) \to_{T \to \infty}
\int_{\cH''} f'' d\mu'' \text{ ).} 
$$ 
By a standard argument we may assume that this is also true if $f''$
is continuous and has a finite limit at infinity; in particular, for
$f \in C_c(\cH)$ we may take $f'' = f \circ \varphi$. Thus by
equivariance we have 
$$
\int_{\cH} f d\mu \longleftarrow \int_{G} f(gM) d\nu_T(g) = \int_{G}
f''(gM'') d\nu_T(g) \longrightarrow \int_{\cH''} f\circ \varphi\, d\mu'', 
$$
and this implies that $\mu = \varphi_*\mu''$. 

The map $\varphi|_{\cL''} : \cL'' \to \cL$ is an affine map of affine
manifolds. In order to show that it is open it suffices to show that
its derivative is surjective at every point $x \in \cL''$. If not,
then there is a neighborhood $\mathcal{U}$ of $x$ in $\cL''$ such that 
$\varphi(\mathcal{U})$ is contained in a proper affine submanifold of $\cL$. Such a proper affine submanifold must have
zero measure for the flat measure class on $\cL$,
i.e. $\mu(\varphi(\mathcal{U}))=0$. By the preceding paragraph this implies
$\mu''(\mathcal{U})=0$ which is impossible. 
\end{proof}

\subsection{The Veech group, lattice surfaces, and periodic
  points}\name{subsec: Veech groups}
An {\em affine automorphism} of a translation surface $M$ is a
 homeomorphism
$\varphi: M \to M$ which is affine in charts. In this case, by connectedness, its
derivative $D\varphi$ is a constant $2 \times 2$ matrix of determinant $\pm 1$. We
denote by $\Aff^+(M)$ the group of orientation-preserving affine automorphisms,
i.e.\ those for which $D \varphi \in G$. We say that $\varphi$ is a
{\em parabolic automorphism} 
if
$D \varphi$ is a parabolic matrix, i.e., is not the identity
but has both eigenvalues equal to 1. 
The
{\em Veech group} of $M$ 
is the image under the homomorphism $D: \Aff^+(M) \to G$ of the group of
orientation-preserving affine automorphisms. We say that $M$ is a {\em
  lattice surface} if its Veech group is a lattice in
$G$. Equivalently, by a theorem  of Smillie (see
\cite{Veech_hayashibara, toronto}), the orbit $GM$ is
closed. Following \cite{HST} we say
that $M$ is a {\em pre-lattice surface} if $\Aff^+(M)$ contains
two non-commuting parabolic automorphisms. 
%
Veech \cite{Veech_alternative}
showed that a lattice surface is a pre-lattice surface, 
justifying the terminology.  A point $x \in M$ is called {\em
  periodic} if its orbit under $\Aff^+(M)$ is finite. 

\subsubsection{Example}
In Lemma \ref{lem: projection closed} we showed that
$\varphi''|_{\cL''} : \cL'' \to \cL$ is an open map. Given that $\cL$
is connected, this leads to the
question of whether $\varphi|_{\cL''}$ is surjective. 
The following example of Alex Wright shows that an
open affine map of orbit-closures need not be surjective. Let
$M \in \cH$ be a lattice surface which admits an involution $\tau$
(e.g. $M$ could be a surface of 
genus 2 and $\tau$ could be the hyper-elliptic involution). Let $\cL = GM$ be the orbit of $M$ (which in this case
coincides with the orbit closure), let $x \in M$ be a non-periodic
point, and let $M'\df (M,x)$ be the surface in $\cH'$ obtained by marking
the point $x$. It was proved in \cite{HST}, and follows easily from
Theorem \ref{EsMiMo}, that $\cL' \df \overline{GM'}$ coincides with
$\varphi'^{-1}(GM)$ (i.e. all surfaces in $GM$ marked at all
nonsingular points). Now let $y \df \tau(x) \neq x$, let $M'' \df
M(x,y)$ be the surface in $\cH''$ obtained by marking $M$ at the two
points $x,y$, let $\cL'' \df \overline{GM''}$, and let $\varphi'':
\cH'' \to \cH'$ be
the affine map which forgets the second marked point. We have 
$$
\cL'' \subset \{(M_0, x_0, y_0) \in \cH'': M_0 \in \cL, \, \tau(x_0)=y_0
\neq x_0\}, 
$$  
since the set on the right-hand side is closed and $G$-invariant. This
implies that $\varphi''(\cL'') \subset \{(M_0, x_0): M_0 \in GM,
\tau(x_0) \neq x_0\}$, and in particular $\varphi''|_{\cL''}$ is not
surjective. However the proof of Lemma \ref{lem: projection closed}
shows that $\varphi''|_{\cL''}$ is open. 

Using one additional marked point one can find similar
examples that show that in general, in Lemma \ref{lem: projection
  closed}, one need not have $\varphi(\cL'') = \cL$.   

\section{Bounded blocking defines closed sets}

Let $M$ be a translation surface with singularity set $\Sigma$, and
let $\widehat{M^2} = \{(x, y) \in (M \sm \Sigma)^2: x \neq y\}.$ If
$Z$ is a topological space and $A \subset B$ are subsets of $Z$, when
we say that {\em $A$ is closed as a subset of $B$}, we mean that $A$ is closed in the
relative topology, i.e. $A = B \cap \overline{A}$. 

\begin{lem}\name{lem: closed sets}
For any fixed integer $n \geq 0$, the following hold:
\begin{enumerate}
\item[(I)]
For a fixed translation surface $M$, the set 
$$F_n(M)  \df \{(x,y) \in \widehat{M^2} : \bc(x, y) \leq n\}$$
is closed as a subset of
$\widehat{M^2}. $
\item[(II)]
For a fixed translation surface $M$, and a fixed nonsingular $x \in M$, the set
$$F_n(M,x) \df \{y\in M \sm (\Sigma \cup \{x\}): \bc(x,y) \leq n\}$$
is closed as a subset of $M \sm (\Sigma \cup \{x\})$.  
\item[(III)]
The set $\cF_n \subset \cH''$ consisting of all surfaces on which the
first and second marked points are finitely blocked of blocking
cardinality at most $n$, is closed in $\cH''$. 
\item[(IV)] For a fixed stratum $\cH$, the set of $M_0 \in \cH$ for
  which any pair $(x, y) \in \widehat{M_0^2}$ satisfies $\bc(x,y) \leq
  n$ is closed in $\cH$. 

\end{enumerate}

Moreover, there is $\ell$ such that
if the set 
\eq{eq: appears in proof}{
\left\{(x,y) \in M^2: \bc(x,y) \leq n \right\}
}
is dense in $M^2$, then
$M$ has the bounded blocking property
with blocking cardinality at most $\ell$.
\end{lem}

\begin{proof}
We will denote a surface in $\cH''$ by
$(M,x,y)$, where $x$ and $y$ are respectively the first and second
marked points on $M$. The topology on $\cH''$ is such that
when $(M_k, x_k, y_k) \to (M,x,y)$, for any parametrized line segment
$\{\sigma(t): t \in [0,1]\}$ on $M$ between $x$ and $y$, for any large
enough $k$ there are parametrized line segments $\{\sigma_k(t): t \in
[0,1]\}$ such that $\sigma_k(t) \to \sigma(t)$ for all $t$ -- see
\cite{MT, Zorich} for details. Here a parameterized line segment is a
constant speed straight line in each chart and does not contain
singular points in its interior.   

We begin with the proof of (III). 
Let $(M_k,x_k,y_k)$ be a sequence that converges to $(M,x,y)$ in
$\cH''$, where $(x_k, y_k)$ belongs to $F_n(M_k)$ for all $k$. Let
$\left\{b^{(1)}_k,
    \dots, b^{(n)}_k\right\} \subset M_k$ be a blocking set for $(x_k,
  y_k)$. Passing to a 
    subsequence, we may assume that $b^{(i)}_k$ converges to a point
    $b^{(i)} \in M$ for
    each $i$. By the above description of the topology of $\cH''$, if
    $\left\{b^{(1)}, \dots, b^{(n)}\right\}$ does not contain $x$ 
    or $y$ then it is a blocking set for
    $(x,y)$ in $M$ and we are done. 

We now discuss the case that some of the $b^{(i)}$ are equal to $x$ or
$y$. We modify the set $\left\{b^{(1)}, \ldots, b^{(n)}\right\}$ as follows.  For
any $i$ for which $b^{(i)}$ is different from both $x$ and $y$, we set
$B^{(i)} =     b^{(i)}$. Suppose $i$ is such that $b^{(i)}=x$. 
    Let $r>0$ be smaller than half the length of the shortest saddle
    connection on $M$. This implies that $r$ is smaller than half the
    distance between $x$ and $y$, and that there is no singularity in the
    ball $B(x,r)$ with center $x$ and radius $r$.

    For $k$ large enough, $B(x_k,r)$ is an embedded flat disk in $M_k$
    that contains $b^{(i)}_k$, and 
    there is a unique trajectory $\delta_{k}^{(i)}$ from $x_k$ to $b^{(i)}_k$ that stays
    within this disk. 
 Let  $B^{(i)}_k$ be the point on $\delta^{(i)}_k$ at distance $r$
    from $x_k$. Passing again to a subsequence we assume that $B^{(i)}_k$ converges to a
    point $B^{(i)}$ in $M$. Note that this point is distinct from $x$
    and $y$ for each such $i$. We repeat this procedure for each $i$
    for which $b^{(i)}$ is equal to either $x$ or $y$, passing at each
    stage to a further subsequence.

    Let us prove that $\left\{B^{(1)}, \dots, B^{(n)}\right\}$ is a blocking set for $(x,y)$ in
    $M$. Let $\sigma$ be a trajectory from $x$ to $y$. We can assume
    without loss of generality that $\sigma$ is simple, i.e. does not
    intersect itself. 
    Let $\sigma_k$ be the segment 
    between $x_k$ and $y_k$ that converges pointwise to $\sigma$.
    If $\sigma_k$ meets one of the $B^{(i)}_k$ for infinitely many $k$, $B^{(i)}$ belongs to
    $\sigma$ and we are done.

    Assume by contradiction that there is an index $i$ such that, for
    infinitely many $k$, $\sigma_k$ meets $b^{(i)}_k$ but not any $B^{(j)}_k$. In
    particular, $b^{(i)}_k$ converges to either $x$ or $y$. Suppose
    for concreteness that it converges to $x$. Since
    $B^{(i)}_k$ does not belong  
    to $\sigma_k$, the subsegment $\sigma'_k$  of $\sigma_k$ between
    $x$ and $b^{(i)}_k$ is not equal to the segment $\delta_{k}^{(i)}$
    defined above. In particular the length of this subsegment is
    bounded below and it converges to a nontrivial subsegment
    $\sigma'$ of $\sigma$, which 
is a loop from $x$
    to $x$. This contradicts the simplicity of $\sigma$, completing
    the proof of (III).

Clearly (III) $\implies $ (I) $\implies $ (II) and (III) $\implies$
(IV). It remains to
prove the final assertion. Suppose $x, y \in M$ and there
are $x_k \to x, y_k \to y$ such that $\bc(x_k, y_k) \leq n$. We need
to prove that $\bc(x, y) \leq \ell$, for some $\ell$ which depends
only on $M$ and $n$. If $x,y$
are distinct nonsingular points, then for large enough $k$ the points
$x_k, y_k$ are also distinct and nonsingular, and the claim follows
from (I). We will consider three cases, adapting the proof of (III) to
each one:

{\em Case 1. } 
$x=y$ is a nonsingular point. We will show that in this case the
previous proof applies and we can take $\ell=n$. 

The segment $\sigma$ from $x$ to $x$ is not
contained in the ball $B = B(x,r)$ appearing in the proof. The only place in the proof of (III) in which
we used that $x \neq y$ is that we needed to know that the subsegment
$\sigma'$ of $\sigma$ constructed in the proof is a proper subsegment
of $\sigma$. In the case $x=y$ there are two subsegments $\sigma'_k,
\sigma''_k$ between $x$ and $b^{(i)}_k$, neither of which is equal to
$\delta_{k}^{(i)}$. In particular each of them leaves the disk $B(x_k,r)$
and hence has length at least $r$. So in the limit they both converge
to nontrivial loops $\sigma', \sigma''$ from $x$ to itself, whose
concatenation is $\sigma$. This gives the desired contradiction to the
simplicity of $\sigma$. 

{\em Case 2.}
$ x \neq y$ and at
least one of them is a singularity. We will show that in this case we
can take $\ell= n(\tau+1)^2,$ where $\tau \pi$ is the maximal cone angle of a
singularity on $M$.  

Assume that $x$ is a singularity, let $r$ be as in
the preceding proof, and let
$\mathcal{U}_1, \ldots, \mathcal{U}_{\tau+1}$ be open convex subsets
of $M$ of diameter less than $r$, such
that $\bigcup \mathcal{U}_s= B(x,r) \sm \{x\}$ and $x$ is in the
closure of each $\mathcal{U}_s$. Such sets exist by
our choice of $\tau$ and $r$, e.g. we may take them to be open half-disks
centered at $x$. If $y$ is also a 
singularity, we similarly choose $\mathcal{U}'_1, \ldots,
\mathcal{U}'_{\tau+1}$ covering $B(y,r) \sm \{y\}$. 

We now choose sequences $x_k^{(s)}$ such that $x_k^{(s)} \in
\mathcal{U}_s$ and $x_k^{(s)} \to_{k \to \infty} x$. If $y$ is also a
singularity we similarly choose sequences $y^{(t)}_k$ which approach
$y$ from within $\mathcal{U}'_t$. We also require that $\bc\left(x_k^{(s)},
y^{(t)}_k\right) \leq n$
for each $k,s,t$. Such sequences exist since \equ{eq:
  appears in proof} is dense. For each choice of $(s,t) \in \{1,
\ldots, \tau+1\}^2$ we perform the procedure explained in the proof of
(III). Namely we take blocking sets $\left\{b_k^{(i,s,t)}: i=1,
  \ldots, n \right\}$ which block all segments between $x_k^{(s)}$ and
$y_k^{(t)}$, pass to subsequences to assume that
$\lim_k b_k^{(i,s,t)}$ exists for each $i,s,t$, and define $B^{(i,s,t)}$
to be this limit if it is distinct from $x$ and $y$. If the limit is
$x$ we modify $b_k^{(i,s,t)}$ by letting $B^{(i,s,t)}_k$ be the
unique point of distance $r$ from $x$ along the continuation of the
unique segment $\delta_{k}^{(i,s,t)}$
which connects $x$ and $b_k^{(i,s,t)}$ and which passes through
$\mathcal{U}_s$. Then we take $B^{(i,s,t)}$ to 
be the limit $\lim_k B_k^{(i,s,t)}$ (passing to subsequences if
necessary). We perform a similar modification if 
$\lim_k b_k^{(i,s,t)}=y.$ This procedure gives us a set 
$$\left\{B^{(i,s,t)}: i \in \{1, \ldots, n\},
  (s,t) \in \{1, \ldots, \tau+1\}^2 \right\},$$ 
which we claim is a blocking set for $x, y$. 

Indeed for each segment $\sigma$  from $x$ to $y$, we can assume that
it approaches $x$ from within $\mathcal{U}_s$ and $y$ from within
$\mathcal{U}'_t$. Then for large enough $k$ there are segments $\sigma_k$ from $x^{(s)}_k$
to $y^{(t)}_k$ which approach $\sigma$ pointwise. Working with these
segments as in the proof of (III), we see that $\sigma$ is blocked by
$B^{(i,s,t)}$ for some $i$.

{\em Case 3.} 
$x=y$ is a singular point. In this case we use both of the arguments
used in Cases 1 and 2. We leave the details to the reader. 
\end{proof}
\begin{cor}\name{cor: closed sets}
Let $\BB_n$ denote the set of surfaces which have the bounded blocking
property, with blocking cardinality at most $n$. Then there is $\ell \in
\N$ such that if $M \in \BB_n$ then $\overline{GM}
\subset \BB_{\ell}$. 
\end{cor}

\begin{proof}
Let us say that $M_0$ is {\em $n$-blocking for distinct nonsingular
  points} if for any $(x,y) \in \widehat{M_0^2}$, $\bc(x,y) \leq
n$. Then the set of such surfaces is closed by Lemma
\ref{lem: closed sets}(IV). Also, if $M''\in \cH''$ has  $x,y$ as the first and second marked
points, and $x(g), y(g)$ are
the first and second marked points on $gM''$, then $\bc(x,y) \leq n$
implies $\bc(x(g), y(g))\leq n$. This implies that the property of
being $n$-blocking for distinct nonsingular points is
$G$-invariant. Since $M \in \BB_n$, $M$ is $n$-blocking for distinct
nonsingular points. Thus any surface in $\overline{GM}$ is also
$n$-blocking for distinct nonsingular points, and the claim follows
from the last assertion in Lemma \ref{lem: closed sets}. 
\end{proof}

A similar argument also shows:
\begin{prop}\name{prop: closed sets2}
Let $M$ be a translation surface, $\xi$ a singular point on $M$ and
$n \geq 0$ an integer. Recalling our convention that singularities on
translation surfaces are labeled, we can use the notation $\xi$ for a
singular point of any other surface in $\cH$. Let $\FF'_n \subset \cH'$ denote the set of
surfaces on which the marked point $y$ satisfies $\bc(\xi, y)
\leq n$. Then $\FF'_n$ is closed in $\cH'$. In particular, $\{y \in M
\sm \Sigma: \bc(\xi, y) \leq n\}$ is closed as a subset of $M \sm
\Sigma.$ 
\end{prop}

\begin{proof}
We repeat the proof of Lemma \ref{lem: closed sets}(III),
replacing everywhere $x$ with $\xi$ and also $x_k$ with $\xi$.  

In this case the set $B=B(\xi, r)$ is a topological disk which is
metrically a finite cover of a flat disk, branched over its center
point $\xi$. Then $B$ is star-shaped with respect to its center
point $\xi$ and it is still the case that there is a unique
straight segment 
from $\xi$ to any point in $B$ which is contained in $B$. We can thus
define the segment $\delta_{k}^{(i)}$ as in the proof of (III), and the same
argument applies. 
\end{proof}

\section{Characterization of the finite blocking property}
In this section we will prove Theorem \ref{thm: BP}. 
A translation surface is \emph{purely periodic} if it is completely
periodic and all
cylinders in such a decomposition have commensurable circumferences. 
The following was proved in \cite{ThierryPP}:

\begin{prop}[Monteil]
\name{prop:bb-pp}
If $M$ has the blocking property then $M$ is purely periodic.
\end{prop}

\begin{proof}[Proof of Theorem \ref{thm: BP}]
The implication \eqref{enum:torus-cover} $\implies$ \eqref{enum:blocking}
is proved in \cite{ThierryBP}, and it is immediate that
\eqref{enum:blocking} $\implies$ \eqref{enum:open-blocking}. 
We first show \eqref{enum:bounded-blocking} $\implies$
\eqref{enum:torus-cover}, that is we assume that $M$ has the bounded
blocking property and we show that it is a torus cover.

Let $\cL \df \overline{GM}$.
By assumption there is $n$ such that $M \in \BB_n$, and by Corollary
\ref{cor: closed sets} this means $\cL$ is contained in $\BB_{\ell}$
for some $\ell$. By
Proposition \ref{prop:bb-pp} this means that every surface in $\cL$
is completely periodic and by Proposition \ref{prop: Wright
  equivalences}, $\cL$ is of cylinder 
rank one. 

Recall that the {\em field of definition} of $\cL$ is the smallest
field such that in any coordinate chart $U$ on $\cH$ given by period
coordinates, the connected components of $U \cap \cL$ are cut out by linear
equations with coefficients in $k$ (see \cite{Wright_field}).  By
\cite[Theorem 1.9]{Wright}, for any
completely periodic surface $M'\in\cL$, and any cylinder decomposition
on $M'$ with circumferences $c_1, \dots, c_r$ the field of definition
$k$ of $\cL$ satisfies $$
k \subset \Q\left[
\left\{
\frac{c_i}{c_j}: i, j = 1, \ldots,
    r\right\}
\right].
$$ 
By Proposition \ref{prop:bb-pp}, any surface in $\cL$ is purely
periodic, so $k =
\Q$. Therefore $\cL$ contains a surface with rational holonomies,
i.e.\ a square-tiled surface $M'$. Since $M'$ 
is square-tiled the holonomy of
absolute periods on $M'$ is a discrete subset of $\C$. Motion in the
$\GREL$ leaf only changes the holonomy of absolute periods by a linear
map, and therefore for any $M$ in $\cL$, the holonomy of absolute
periods is discrete, i.e., any $M\in\cL$ is a torus cover. 
This proves \eqref{enum:bounded-blocking} $\implies$ \eqref{enum:torus-cover}.

Now we prove \eqref{enum:open-blocking} $\implies$ \eqref{enum:bounded-blocking}.
We have an open set $U_1$ in $M \times M$ consisting of pairs of
points on $M$ blocked from each other by finitely many points, that
is, $U_1 \cap \widehat{M^2} \subset \bigcup_n F_n(M)$. Each $F_n(M)$
is closed as a subset of $\widehat{M^2}$ by Lemma
\ref{lem: closed sets}(I), so
by Baire category, there is $n$ such that $F_n(M)$ contains an open
set $U_2$. Each pair of points $(x,y)$ in $U_2$ defines
a surface in $\cH''$, namely $M'' = (M,x,y)$. Let $\cL(M'')
\df \overline {GM''} \subset \cH''$. 
By Theorem \ref{EsMiMo}, 
$\widehat{\cL}(M'')$ is a linear manifold of even dimension contained in $\cF_n$ and the
collection of such linear submanifolds is countable. By Lemma
\ref{lem: closed sets}(III), $\widehat{\cL}(M'') \subset \cF_n$. 

The fiber $\phi^{-1}(M)$ is a linear submanifold of $\cH''$
identified with $\widehat{M^2}$. Therefore
for any $M''$, $\Omega(M'') \df \phi^{-1}(M) \cap \widehat{\cL}(M'')$ is also a
linear submanifold, and its dimension is $0,2$ or
$4$. We have covered $U_2$, an open subset of a four-dimensional manifold,
by countably many linear manifolds of dimensions at most four. By
Baire category, there is $M''$ for which $\Omega(M'')$ is a linear manifold of
dimension four. Since $\phi^{-1}(M)$ is connected, it coincides with
$\Omega(M'')$. 

We have proved that $$\phi^{-1}(M) = \Omega(M'')
\subset \widehat{\cL}(M'') \subset \mathcal{F}_n;$$
that is, any two distinct nonsingular points in $M$ are of blocking
cardinality at most $n$. Applying the last assertion of Lemma
\ref{lem: closed sets}, we see that $M$ has the bounded blocking property. 
\end{proof}
\section{Illumination}\name{section: illumination}
In this section we will study some illumination problems. 
Recall
that two points $x, y$ on a translation surface $M$ do
not illuminate each other if and only if 
they are
finitely blocked with blocking cardinality zero. Also recall that
$p_1, p_2$ denote the projections onto the first and second factors of
$M \times M$. The following result is the
main result of this section. 

\begin{thm}\name{thm: more general}
Let $M$ be a translation surface, let $n$ be a non-negative integer. 
Then:
\begin{enumerate}
\item[(i)]\name{item: 1}
For any $x \in M$, the set $
\{y \in M: \bc(x,y) \leq n\}$
is either finite or contains $M \sm (\Sigma \cup \{x\})$.  
\item[(ii)]\name{item: 2}
The set $\{(x,y) \in M^2: \bc(x,y) \leq n\}$ either contains
$\widehat{M^2}$, or is a finite union of 0 and 2
dimensional linear submanifolds of $M \times M$. The 2-dimensional linear
submanifolds are of one of the following forms: $F \times M$, $M
\times F$, where $F \subset M$ is finite, or 
a translation surface $S$ embedded affinely in $M \times M$,
where for $i=1,2$, $\tau_i = p_i|_S: S \to M$ is a finite-degree covering map
such that $\tau_2 \circ \tau_1^{-1}$ is a multiplication by a 
scalar $\lambda$ satisfying $\lambda^2 \in \Q$. 
\end{enumerate}
\end{thm}

\begin{proof}[Theorem \ref{thm: more general} implies Theorem \ref{thm:
  illumination}]
We apply Theorem \ref{thm: more general} with $n=0$. It is clear that
the second alternative in (i) 
cannot hold, since for any
$x$, all nearby points 
illuminate $x$. Also, in (ii), 
the cases $F \times M$
and $M \times F$ do not arise, since any point illuminates some other
point. 
\end{proof}

\begin{proof}[Proof of Theorem \ref{thm: more general}]
Keep the notation of \S \ref{subsec: adding points} and Lemma
\ref{lem: closed sets}. We will first prove (i) 
in case $x$ is a
regular point of $M$. Let $M' \in \varphi'^{-1}(M) \subset \cH'$
denote the surface with first marked point at $x$. 
We need to show that 
$$A \df \left\{y \in M \sm (\Sigma \cup \{x\}): \bc(x,y) \leq n \right\},$$
which we may identify with 
$\cF_n \cap \varphi''^{-1}(M')$, is either finite or coincides with
$\varphi''^{-1}(M')$.  
Let us assume $A \varsubsetneq \varphi''^{-1}(M')$. Since $\cF_n$ is closed and
$G$-invariant, $A$  is 
a union of at most countably many linear manifolds, which are of the
form $\cL(M''_0) \df \overline{GM''_0}$ for $M''_0 \in A$. For each
$M''_0$, $\cL(M''_0) \cap \varphi''^{-1}(M')$ is a linear manifold of
dimension 0 or 2 by Theorem \ref{EsMiMo}. If the dimension were 2, 
$A$ would coincide with the fiber $\varphi''^{-1}(M)$ by
connectedness, and hence each $\cL(M''_0)$ is finite. To conclude the
proof of (i) 
we need to show that in fact there are only
finitely many distinct sets $\cL(M''_0)$. If there were
infinitely many this would mean that
$\cF_n$ contains infinitely many $G$-orbit-closures, so by Theorem
\ref{EsMiMo} they would accumulate on an orbit-closure of greater
dimension, also contained in $\cF_n$. By Lemma \ref{lem: projection closed}, each
$\cL(M''_0)$ projects onto $\cL$, so the only way for the dimension to
increase would be in the direction of the fiber $\varphi''^{-1}(M')$;
that is, once again we would have that $A = \varphi''^{-1}(M')$,
contrary to assumption. 

In case $x = \xi$ is a singularity we repeat the argument, using Proposition
\ref{prop: closed sets2} instead of Lemma \ref{lem: closed sets},
$\cF'_n$ instead of $\cF_n$, $\varphi'$ instead of $\varphi''$ and
$\cH'$ instead of $\cH''$. We leave the details to the reader.  

The proof of (ii) is similar. Suppose that 
$$\widehat{M^2} \not \subset A \df  \{(x, y) \in M^2: \bc(x,y) \leq n\}.$$
Applying Theorem \ref{EsMiMo} as in the preceding paragraph
we see that $A$ is the union of finitely many 0-dimensional and
finitely many 2-dimensional linear manifolds. We need 
to show that all of the 2-dimensional manifolds have the stated
form. This follows from arguments of \cite{HST}, but we give an
independent argument for completeness. 

Let $N \subset (M \sm \Sigma)^2$ be a
2-dimensional linear manifold in $A$. By
Theorem \ref{EsMiMo}, $N$ is $\C$-linear, i.e., in the translation charts of $M \times
M$, a neighborhood of $N$ is the set of solutions
of an equation of the  form 
\eq{eq: equation of plane}{
a z_1+bz_2 =0
}
(up to a translation). Moreover $N$ is
defined over $\R$ so we can take $a, b \in \R$. If $a=0$ then any
connected component of $N$ is of the form $M \times \{x\}$ for some $x
\in M$, i.e. $N$ is of the form $M \times F$. Similarly if $b=0$ then
$N$ has the form $F \times N$. Now we consider the case when $a,b$ are
both nonzero.

Since the transition maps
for the translation atlas are translations, $a$ and $b$ can actually
be taken to be independent of the neighborhood, and the Cartesian product translation
structure on $M^2$, restricted to $N$, endows $N$ with a natural
structure of a translation surface (see \cite[\S3]{HST} for more
details), where $N$ is locally modelled on the plane \equ{eq: equation
  of plane}. Since $a$ and $b$ are both nonzero, each of the
projections $\tau_i = p_i|_N$ has a nonsingular derivative, so by
connectedness, each $\tau_i$ is a finite covering map. 

The plane
\equ{eq: equation of plane} can be identified with $\C$ in many ways
and thus the translation surface structure on $N$ is only naturally defined up to
a scalar multiple. 
However, for any fixed choice of translation structure on $N$, each of
the maps $\tau_i$ is the composition of a dilation and a translation
covering. Let $k_i$ be the degree of the covering map $\tau_i$, and
let $\lambda_i$ be the associated dilation. The choice of the
$\lambda_i$ depends on a choice of the translation structure on $N$,
but since the derivative of $\tau_2 \circ \tau_1^{-1} $ is the map $z_1
\mapsto -\frac{a}{b}z_1$, we have $\lambda \df \frac{\lambda_2}{\lambda_1} =
-\frac{a}{b}$. We can compute the area of $N$ using each of the
maps $\tau_i$, to obtain 
$$
\mathrm{area} (N) = \frac{ k_i}{\lambda_i^2} \mathrm{area}(M). 
$$ 
Comparing these formulae for $i=1,2$ we see that $\lambda^2 = \left(\frac{a}{b}
\right)^2 = \frac{k_2}{k_1} \in \Q$. 
\end{proof}

\section{Examples and questions}\name{section: examples} 
Let $T$ be the standard torus, obtained from the unit square $[0,1]^2$
by gluing opposite sides to each other by translations. 
%
%
Denote by $\pi$ the projection from $\R^2$ to $T$.
For any nonzero integer $n$, notice that the map $\R^2 \to \R^2$,
$x\mapsto n x$ descends to a map $m_n: T \to T$ which multiplies
both components by $n$ in $\R/\Z$, and is therefore $n^2$ to $1$.
We describe blocking cardinalities of pairs of points in $T$ and
blocking sets realizing them.

\begin{lem}
\name{torus-distinct}

\begin{itemize}
\item[(a)]
If $x$ and $y$ are distinct points on $T$, their
blocking cardinality is $\bc(x,y) = 4$.
\item[(b)]
It is realized by the blocking set $B(x,y) = m_2^{-1}(x+y)$, which contains
the midpoint of any geodesic from $x$ to $y$.
\item[(c)]
This is the unique blocking set of size $4$.
\end{itemize}

\end{lem}

\begin{proof}
Let $\til x$, $\til y$ denote points
in $\R^2$ which project to $x$, $y$ on $T$. Let $u = (1,0)$,
$v = (0,1)$, $w = (1,1)$. The four segments from $\til y$ to the four
points 
$\til x$, $\til x+u$, $\til x+v$, $\til x+w$ (four corners of a unit square)
project to segments with disjoint interiors on $T$, so at least $4$ points 
are required to block the pair $(x,y)$. 
On the other hand, any line segment in $T$ from $x$ to $y$
is the projection of a line segment  in $\R^2$ from $\til x$ to $\til y + a u + b v$
with $a$ and $b$ in $\Z$. Such a segment has midpoint
$\frac{1}{2}(\til x + \til y + a u + b v)$. 
This midpoint in $\R^2$ projects to one of the points $\frac{1}{2}(x+y)$,
$\frac{1}{2}(x+y+u)$, $\frac{1}{2}(x+y+v)$, $\frac{1}{2}(x+y+w)$,
which are the four points in $T$ comprising $m_2^{-1}(x+y)$. This
proves that the set $B(x,y)$ is a blocking 
set and that $\bc(x,y) \leq 4$. So (a) and (b) are proved. 

We now prove (c).
We saw that the four segments from $\til y$ to $\til x$, $\til x+u$,
$\til x+v$, 
$\til x+w$ project to segments on $T$ with disjoint interiors, so a blocking 
set for $(x,y)$ must contain at least a point in each of them.
Consider the segment from $\til y+v$ to $\til x+u$. The only
intersection of its projection to $T$ with the interiors of our four segments
is its midpoint $m$ which is also the midpoint of the segment from $y$ to 
$y+w$. So a blocking set not containing $m$ would need to contain 
at least $5$ points. Similar reasoning proves the other three points in 
the proposed set $B(x,y)$ have to be in a blocking set of cardinality $4$.
\end{proof}

The following two lemmas extend this description to configurations
blocking a point from itself, and describe larger blocking sets on
$T$. They are proved by similar arguments and we leave the details to
the reader. 

\begin{lem}
\name{torus-same}
\begin{itemize}
\item[(a)]
If $x = y$, then the blocking cardinality is $\bc(x,x) = 3$.
\item[(b)]
It is realized by the blocking set $B(x,x) = m_2^{-1}(2 x) \sm \{x\}$,
which is the set of midpoints of all primitive geodesics from $x$ to
$x$. This blocking set  
can also be described as $B(x,x) = x + B_0$ where $B_0 = B(0,0) = 
m_2^{-1}(0)\sm \{0\}$.
\item[(c)]
This is the unique blocking set of size $3$.
\end{itemize}
\end{lem}

\begin{lem}
\name{torus-distinct-more}
\begin{itemize}
\item[(a)]
Let $n$ and $a$ be relatively prime integers with $1 \leq a < n$.
For any pair of points $(x,y)$ with $x \neq y$, 
the set $B = m_n^{-1}(a x + (n-a) y)$ is a blocking set of cardinality $n^2$ 
for the pair $(x,y)$. It contains the point located $a/n$ of the way along each 
line segment from $x$ to $y$ on $T$.
\item[(b)]
Let $n \geq 2$ be an integer. For the pair of points
$(x,x)$ with $x = 0$, the set 
$$B_0 = m_n^{-1}(0)\sm\{0\} = \{ (a/n,b/n) : 
0 \leq a < n,\quad 0 \leq b < n,\quad (a,b) \neq (0,0) \}$$
is a blocking set of cardinality $n^2 - 1$. 

For the pair of points 
$(x,x)$ with $x \neq 0$, the set $B = x + B_0$
is a blocking set of cardinality $n^2 - 1$, also equal to
$m_n^{-1}(n x)$.
\end{itemize}
\end{lem}

We will use these computations to compute blocking configurations on
brached covers of $T$. Recall that if $M \to T$ is a branched
translation cover, a singularity of $M$ corresponds to a ramification
point of the cover, and if the angle at a singularity $x$ is $2\pi k$ then
$k$ is called the {\em ramification index} of $x$. 

\begin{lem}
\name{degree-d-torus-covering}
Suppose $M$ is a torus cover of degree $d$, with arbitrary branch locus and 
ramification type, and let $p: M \to T$ denote the covering map. 

\begin{itemize}
\item[(a)]
For a pair $(x,y)$ of points of $M$ such that $p(x) \neq p(y)$, 
if $B'$ is a blocking set for $(p(x),p(y))$ on $T$, then 
$B = p^{-1}(B')$ is a blocking set for $(x,y)$, of cardinality
at most $d$ times that of $B'$, with equality when $B$ contains no 
zero of $M$, i.e.\ no ramification point of $p$.
\item[(b)]
In particular,
\begin{itemize}
\item
for almost every pair $(x,y)$ of points of $M$, $\bc(x,y) \leq 4 d$.
\item
for pairs $(x,y)$ of points of $M$, such that the set $B(p(x),p(y))$
contains branch points of $p$, the bound above is decreased by the sum 
of the ramification indices of the ramification points above these branch 
points.
\end{itemize}
\item[(c)]
For a pair of points $(x,y)$ on $M$ such that $p(x) = p(y)$ (whether 
$x=y$ or not), $p^{-1}(B(p(x),p(x)))$ is a blocking set, so that 
$\bc(x,y) \leq 3 d$.
As above, when $B(p(x),p(y))$ contains branch points of $p$, the 
bound is decreased by the sum  of the ramification indices of the
ramification points above these branch points.
\end{itemize}

\end{lem}

\begin{proof}
Both (a) and (b) are easy, and (c) follows from the following observation.
When $p(x) = p(y)$, any geodesic path $\gamma$ from $x$ to $y$ 
projects to a geodesic $\gamma'$ from $p(x)$ to itself, possibly non 
primitive. Considering the restriction of the geodesic $\gamma$,
if $\gamma'$ is not primitive, to its initial part until it first
reaches a point projecting to $p(x)$, we see that (c) holds.
\end{proof}




\subsection{Example 1} 
The following example shows that quite general maps $\tau_1, \tau_2$
may arise in Theorem \ref{thm: more general}. 
\begin{prop}
\name{torus-a-b}
Let $a$, $b$ be positive integers with $\gcd(a,b) = 1$,
let $n = a + b$, and let 
$$X = \{(- a x, b x) : x \in T\}.$$ Also let $p: M \to T$ be a
translation cover with branching locus $m_n^{-1}(0)$, and
nontrivial ramification at each pre-image of each branch point, and
let 
$$
Y = (p \times p)^{-1}(X).$$ 
Then any pair of points in $Y$ do not illuminate each other. 
\end{prop}

\begin{proof}
For $x \in \R^2$, the point $0$ is $a/n$ along the geodesic in $\R^2$ from 
$-ax$ to $bx$. Thus, according to Lemma \ref{torus-distinct-more}, the
set $B = m_n^{-1}(0)$ is a common blocking set, of  
cardinality $n^2$, for all pairs
of points in $X$. Thus the statement follows from Lemma \ref{degree-d-torus-covering}. 
\end{proof}

\subsection{Example 2} 
The following examples show that the map $\tau_2 \circ \tau_1^{-1}$
could be a translation. Let $M=T$ be the torus, and consider 
$$
N = \{(x,y) \in M^2: \bc(x,y) \leq 3\}.
$$
Then according to Lemma
\ref{torus-same}, $N$ contains the diagonal $\{(x,x): x \in M\}$ but
according to Lemma \ref{torus-distinct}, $N \neq M^2$. Therefore the
diagonal is one of the linear submanifolds appearing in Theorem
\ref{thm: more general}, and we can have $\tau_2 \circ \tau_1^{-1} =
\mathrm{Id}$. 

\begin{figure}
\begin{tikzpicture}
\useasboundingbox (0,0) rectangle (6,6.5);
\draw (0,0)
  -- node [above=-2pt] {\includegraphics[scale=0.1]{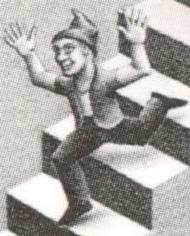}} ++(2,0)
  -- ++(2,0) -- ++(0,2)
  -- ++(2,0) -- ++(0,2)
  -- node [above=-2pt] {\includegraphics[scale=0.3]{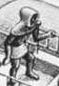}} ++(2,0) -- ++(0,2)
  -- node [above=-2pt] {\includegraphics[scale=0.1]{1.jpg}} ++(-2,0)
  -- ++(-2,0) -- ++(0,-2)
  -- ++(-2,0) -- ++(0,-2)
  -- node [above=-2pt] {\includegraphics[scale=0.3]{2.jpg}} ++(-2,0) -- ++(0,-2)
  -- cycle;
 \draw [very thick] (0.88,0.05) -- (0.92,0.13) -- (0.925,0.2);
 \draw [very thick] (0.925,0.2) -- (0.975,0.14) -- 
 (0.94,0.08);
 \draw [very thick] (0.925,0.19) -- (0.99,0.35);
 \draw [very thick] (0.985,0.33) -- (0.91,0.26) -- (1,0.23);
 \draw [fill=black] (1,0.36) circle (0.9pt);
 \begin{scope}[xshift=6cm,yshift=8cm]
 \draw (0.9,0.1) -- (0.925,0.2) -- (0.96,0.15) -- (0.94,0.12);
 \draw (0.925,0.2) -- (0.95,0.28);
 \draw (0.95,0.28) -- (0.925,0.24) -- (0.96,0.23);
 \draw [fill=black] (0.96,0.29) circle (0.4pt);
\end{scope}
\foreach \x in {0,2,4,6}
  {\draw [very thin,fill=white] (\x,\x) circle (2.2pt);
 \draw [fill=black!70] (\x+2,\x) circle (2pt);}
  ;
\foreach \x in {0,2,4}
  {\draw [very thin,fill=white] (\x+4,\x) circle (2.2pt);
  \draw [fill=black!70] (\x,\x+2) circle (2pt);}
  ;
\end{tikzpicture}
\caption{The Escher double staircase. Sides marked with identical
 stair-climbers are identified; unmarked sides are identified with
  the corresponding opposite sides. } 
\label{fig: Escher} 
\end{figure}

Similar examples in which $\tau_2 \circ \tau_1^{-1}$ is a nontrivial translation
can be obtained by taking $M$ to be a cyclic cover of $T$, for example
the Escher staircase (see Figure \ref{fig: Escher}). 
This surface admits a degree 3 cover $p:M \to T$ and it has a nontrivial
translation automorphism $D: M \to M$ moving one step 
up the ladder. Let $x$ and $y$ be any two points such that
$D(x)=y$. Then $p(x)=p(y),$ and 
according to Lemma \ref{degree-d-torus-covering}(c), $\bc(x,y) \leq
9$. It is not hard to find  an explicit pair of points $x,y$ for which
$\bc(x,y)>9$. This shows that if we take this surface $M$ and $n=9$,
then we can have a subsurface $N$ for which $D=\tau_2 \circ \tau_1^{-1}$
is a translation automorphism.

\subsection{Example 3.}\name{related example2}
Using the torus and Lemmas \ref{torus-distinct} and \ref{torus-same}
we easily find sequences $x_k \to x, y_k \to y$ for 
which $\bc(x,y) < \lim_k \bc(x_k, y_k)$, i.e. the blocking cardinality
is not continuous. 
The following example shows that blocking cardinality it is not even 
lower semi-continuous, i.e. it may increase  when taking limits. It
also shows that in Lemma \ref{lem: closed sets}(I), we cannot replace
$\widehat{M^2}$ with $M^2$. 
 
Let $M$ be a surface in $\cH(2)$. Then $M$ admits a hyper-elliptic
involution $h$, whose set of fixed points consists of the unique
singularity $\xi$, and $5$ non-singular Weierstrass points. We claim
that whenever $h(x) = y, x \neq y$, we have $\bc(x,y) \leq
5$. Indeed, in this case, the action of $h$ swaps $x$ and $h(x)$, and
acts by rotation by $\pi$. So $h$ maps any segment between $\sigma$
between $x$ and $y$ to another segment from $x$ to $y$, of the same
length and in the same 
direction. Since $x$ and $y$  are distinct regular points there is only one such
segment, i.e. $h$ maps $\sigma$ to itself, reversing the orientation
on it. So its midpoint must be
fixed by $h$, that is, the Weierstrass points form a blocking set for
the pair $(x,y)$. 

On the other hand by constructing explicit disjoint segments, it is
not hard to show that $\bc(\xi, \xi) \geq 9$. For example we can
present $M$ as the union of four parallelograms (an $L$-shaped
presentation), and the diagonal and edges of these parallelograms
contain 9 disjoint segments from $\xi$ to $\xi$. Now taking $x_k \to
\xi,$ we have $y_k = h(x_k) \to \xi$, and 
$$
5 \geq \lim_k \bc(x_k, y_k), \ \ \ \bc(\lim_k x_k, \lim_k y_k) = \bc (\xi,
\xi) \geq 9.
$$

\subsection{Questions}
1. Let $N \subset M \times M$ be a 2-dimensional linear submanifold as in Theorem \ref{thm:
  more general}(ii), and let $\lambda_1, \lambda_2$ be the derivatives of
the translation maps $\tau_1, \tau_2$. The quotient $\lambda = \lambda_1/\lambda_2$ is called
the {\em slope} of $N$.  
In Example 1 the slope is 1, in Example 3 the slope is $-1$, and in Example 2 the slope can be an
arbitrary negative rational number. It would be interesting to know
whether other slopes are possible. In particular, do the cases 
$\lambda=0, \, \lambda=\infty$ actually arise in connection with blocking
configurations? Do positive rational slopes arise, except for $\lambda=1$? 

2. More generally, suppose $N \subset
M \times M$ is an
embedded translation surface for which the maps $\tau_i:N \to M$  are
the composition of a dilation and a translation, and let $\lambda$ be
the derivative of the composition $\tau_2 \circ \tau_1^{-1}$. In the
proof of Theorem \ref{thm: more general} we showed that $\lambda^2 \in \Q$. Is it
possible that $\lambda$ is irrational?

3. In the last assertion of Lemma
\ref{lem: closed sets}, can we take $\ell = n$? Example 4 shows that
this does not follow from a simple continuity argument.

\bibliographystyle{alpha}
\bibliography{translation}

\end{document}